\documentclass[]{article}
\usepackage{graphicx}
\usepackage{graphics}
\usepackage{amssymb,amsmath,amsthm,epsfig}
\newtheorem{theorem}{Theorem}
\newtheorem{lemma}{Lemma}
\newtheorem{corollary}{Corollary}

\newtheorem{proposition}{Proposition}

\newtheorem{remark}{Remark}
\newcounter{rea}
    \newcommand{\K}{\mathbf K}
\newcommand{\E}{\mathbf E}
\newcommand{\epsn}{\varepsilon_n}

\newcounter{rek}
\begin{document}


%
\begin{center}
{\large {\bf Spectral Decay of Time and Frequency Limiting Operator.}}\\
\vskip 1cm Aline Bonami $^a$ and Abderrazek Karoui $^b$ {\footnote{Corresponding author,\\
This work was supported in part by the  French-Tunisian  CMCU 10G 1503 project and the
Tunisian DGRST  research grants 05UR 15-02 and UR 13ES47. Part of this work was done while the first author was visiting the Faculty of Sciences of Bizerte, University of Carthage.
The two authors have also benefited from the program ``Research in pairs " of the CIRM, Luminy, France. }}
\end{center}
\vskip 0.5cm {\small
$^a$ F\'ed\'eration Denis-Poisson, MAPMO-UMR 7349,  Department of Mathematics, University of Orl\'eans, 45067 Orl\'eans Cedex 2, France.\\
\noindent $^b$ University of Carthage,
Department of Mathematics, Faculty of Sciences of Bizerte, Tunisia.}\\
Emails: aline.bonami@univ-orleans.fr ( A. Bonami), abderrazek.karoui@fsb.rnu.tn (A. Karoui)\\

\noindent{\bf Abstract}--- For fixed $c,$ the  Prolate Spheroidal Wave
Functions (PSWFs) $\psi_{n, c}$ form a basis with remarkable properties
for the space of band-limited functions with bandwidth $c$. They have
been largely studied and used after the seminal work of D. Slepian, H. Landau and H. Pollack.
Many of the PSWFs applications rely heavily of the behavior and the decay rate of 
the eigenvalues $(\lambda_n(c))_{n\geq 0}$ of the  time and frequency limiting operator, which  we denote by $\mathcal Q_c.$
 Hence, the issue of the accurate
estimation of the spectrum of this operator has attracted a considerable interest, both in numerical and
theoretical studies. In this work, we give an explicit integral approximation formula  for these eigenvalues. This approximation holds
true starting from the plunge region where  the spectrum of $\mathcal Q_c$ starts to have a fast decay. 
As a consequence of our explicit approximation formula, we give a precise description of the super-exponential decay rate 
of the $\lambda_n(c).$ Also, we  mention that the described approximation scheme provides us with fairly  accurate 
approximations of the $\lambda_n(c)$ with  low computational load, even for very large values of the parameters $c$ and $n.$  
Finally, we provide the reader with some numerical examples  that
illustrate the different results of this work.\\


\section{Introduction}

For a given value $c>0$, called the bandwidth, PSWFs $(\psi_{n,c}(\cdot))_{n\geq 0}$ constitute an orthonormal basis of $L^2([-1, +1]),$ an orthogonal system of $L^2({\bf R})$ and an
orthogonal  basis of  the Paley-Wiener space $B_c,$ given by $B_c=\left\{ f\in L^2({\bf
R}),\,\, \mbox{Support\ \ } \widehat f\subset [-c,c]\right\}.$
Here, $\widehat f$ denotes the Fourier transform of $f$. They are eigenfunctions of the
  compact integral operators
 $\mathcal F_c$ and $\mathcal Q_c= \frac c{2\pi}\mathcal F_c^*\mathcal F_c$, defined  on $L^2([-1,1])$ by
\begin{equation}\label{eq1.1}
 \mathcal F_c(f)(x)= \int_{-1}^1
e^{i\, c\, x\, y} f(y)\, dy,\quad \mathcal Q_c(f)(x)=\int_{-1}^1\frac{\sin c(x-y)}{\pi (x-y)}\, f(y)\,
dy. \end{equation}
Since   the
operator $\mathcal F_c$ commutes with the Sturm-Liouville operator $\mathcal L_c$,
\begin{equation}\label{eq1.0}
\mathcal L_c(\psi)=-\frac{d}{d\, x}\left[(1-x^2)\frac{d\psi}{d\,x}\right]+c^2 x^2\psi,
\end{equation}
 PSWFs $(\psi_{n,c}(\cdot))_{n\geq 0}$ are also eigenfunctions of $\mathcal L_c$. They are ordered in such a way that the corresponding eigenvalues of $\mathcal L_c$, called $\chi_n(c)$, are strictly increasing. Functions $\psi_{n,c}$ are restrictions to the interval $[-1, +1]$ of real analytic functions on the whole real line and eigenvalues $\chi_n(c)$ are values of  $\lambda$ such that the equation $\mathcal L_c \psi=\lambda \psi$ has a non zero bounded solution on the whole interval.

PSWFs have been introduced  by D. Slepian, H. Landau  and  H. Pollak
\cite{LW, Slepian1, Slepian3, Slepian2} in relation with signal processing. 
 For a detailed review on   properties,  numerical computations, asymptotic results and
first applications of the PSWFs, the reader is referred to  recent books on the subject, \cite{Logan}, \cite{Osipov3}.

By Plancherel identity, PSWFs are normalized so that 
\begin{equation}\label{eqq1.4}
\int_{-1}^1 |\psi_{n,c}(x)|^2\, dx = 1,\quad \int_{\mathbb R}
|\psi_{n,c}(x)|^2\, dx =\frac{1}{\lambda_n(c)},\quad n\geq 0.
\end{equation}
Here, $(\lambda_n(c))_n$ is the infinite sequence of  the
eigenvalues of $\mathcal Q_c,$ also arranged in the decreasing order $1>
\lambda_0(c)> \lambda_1(c)>\cdots>\lambda_n(c)>\cdots.$ We call $\mu_n(c)$ the eigenvalues of $\mathcal F_c$. They are given by
$$\mu_{n}(c)=i^n\sqrt{\frac {2\pi}c \lambda_n(c)}.$$
Also, we adopt the  sign normalization of the PSWFs,  
\begin{equation}\label{eeqq1.4}
\psi_{n,c}(0) > 0\mbox{\ \  for even\ \ } n,\quad \, \psi'_{n,c}(0) > 0,\mbox{\ \ for odd \ \ }  n.
\end{equation}

\medskip

One of the main issues that we discuss here is the decay rate of the eigenvalues $\lambda_n(c)$. This decay rate plays a crucial role in most of the 
various concrete applications of the PSWFs.   In this direction, one knows their asymptotic behaviour for $c$ fixed, which has been  given in 1964 by Widom, see \cite{Widom}.
\begin{equation}\label{Widom_asymptotic}
 \lambda_n(c)\sim  \left ( \frac {e c} {4(n+\frac 12)} \right)^{2n+1}=\lambda_n^W(c).
 \end{equation}
This gives the exact decay for $n$ large enough, but one would like to have a more precise information in terms of uniformity of this behaviour, both in $n$ and $c$. On the other hand, Landau has considered  the value of the smallest integer $n$ such that $\lambda_n(c)\leq 1/2$ in \cite{Landau}. More precisely, if we note $c^*_n$ the unique value of $c$ such that $\lambda_n(c)=1/2$, then he proves that
 \begin{equation}
 \label{validityC0intro}
 \frac{\pi}{2}(n-1)\leq c^*_n\leq \frac{\pi}{2} (n+1)\quad\quad \lambda_n(c^*_n)=\frac{1}{2}.
 \end{equation}
So, for $c$ fixed, we almost know when $\lambda_n(c)$ passes through the value $1/2$.  Landau and Widom have also described the asymptotic behaviour, when $c$ tends to $\infty$, of the distribution of the eigenvalues $\lambda_n(c)$.

The search for more precise estimates of the $\lambda_n(c)$ has  attracted a considerable interest, both in numerical and theoretical studies.  We try here to give approximate values for $\lambda_n(c)$ for $c\leq c^*_n$, with some uniformity in the quality of approximation. We rely on the exact formula for the eigenvalues
$\lambda_n(c),$ given by integrating the following differential equation,
see \cite{Fuchs},
\begin{equation}\label{eq1lambda_n}
\partial_{\tau} \ln \lambda_n(\tau)=\frac{2|\psi_{n,\tau}(1)|^2}{\tau}.
\end{equation}
over the  interval $(c,c^*_n).$ This is different from the classical way to estimate the $\lambda_n(c),$ via
equation \eqref{eq1lambda_n}, where the integration is rather done on the interval $(0,c),$ see for example \cite{Rokhlin, Slepian3, Xiao}. 
Consequently, we are mainly interested in the behavior, as well as in  accurate and fast schemes of approximation, 
of the $\lambda_n(c)$, given by the  formula 
\begin{equation}
 \label{eqq2lambda_nintro}
 \lambda_n(c)=\frac{1}{2} \exp\left(-2\int_c^{c^*_n} \frac{(\psi_{n,\tau}(1))^2}{\tau} \, d\tau \right).
  \end{equation}

We use our recent works \cite{Bonami-Karoui1, Bonami-Karoui2}  to estimate the value $\psi_{n,\tau}(1)$. In the first paper it is proved that $|\psi_{n,\tau}(1)|\leq 2\chi_n(\tau)^{1/4}$, which is not sufficient to find a sharp estimate for all values $c$.  The approximation given in the second paper leads to a second estimate of $\psi_{n,\tau}(1),$  valid for $\frac{\pi n}{2}-c$ larger than some multiple of $\ln n$. 
Based on this second estimate, we  define $\widetilde { \lambda_n}(c)$ as
\begin{equation}\label{decay22}
\widetilde { \lambda_n}(c) =  \frac{1}{2} \exp\left(-\frac{\pi^2(n+\frac 12)}{2} \int_{\Phi\left(\frac{2c}{\pi (n+\frac 12)}\right)}^1 \frac{1}{t(\mathbf E(t))^2}\, dt \right).\end{equation}
Here $\mathbf E$ is the    elliptic  integral of the second kind given by \eqref{elliptic_integrals} and the function $\Phi$ is the inverse of the function $t\mapsto \frac{t}{\E(t)}$.
 We  prove that $\widetilde { \lambda_n}(c)$ is comparable with $ \lambda_n(c)$ up to some power of $n$. This  is stated  in the following theorem, which is the main result of this paper.
\begin{theorem}\label{Required00}
There exist three constants $\delta_1\geq 1, \delta_2, \delta_3, \geq 0$ such that, for $n\geq 3$ and $c\leq \frac{\pi n}2$,
\begin{equation}\label{required00}
\delta_1 ^{-1}n^{-\delta_2}\left(\frac c{c+1}\right)^{\delta_3}\leq \frac{\widetilde { \lambda_n}(c)}{ \lambda_n(c)}\leq \delta_1 n^{\delta_2}\left(\frac c{c+1}\right)^{-\delta_3}.
\end{equation}
 \end{theorem}

Let us explain, roughly speaking, why Legendre  elliptic  integrals are involved here. The Sturm-Liouville equation $\mathcal L_c \psi=\chi_n(c) \psi$ can be rewritten as $$
\psi''-\frac{2x}{1-x^2}\psi'+\chi_n(c)\frac{1-q  x^2}{1-x^2}\psi=0,
$$
with $q=c^2/\chi_n(c)$. Under the assumptions on $n$ and $c$, the value $q$ may be seen as a parameter. We assume that  $q<1$ and  proceed to a WKB approximation of the solution $\psi_{n, c}$, but not directly: the equation is transformed into its normal form $U''+(\chi_n + \theta)U=0$ through the Liouville transformation given by the change of functions $\psi=[(1-x^2)(1-qx^2)]^{-1/4}U$ and the change of variable $x\mapsto\int_0^x\sqrt{\frac{1-qt^2}{1-t^2}} dt$. We recognize in this change of variable   the incomplete Legendre elliptic  integral of the second kind, while the $L^2-$norm of the factor $[(1-x^2)(1-qx^2)]^{-1/4}$ can be written in terms of the complete Legendre elliptic  integral of the first kind. This has been exploited by many authors (for instance \cite{Dunster}, \cite{Xiao}, \cite{Osipov}). We refer to our recent work \cite{Bonami-Karoui2} for the approximation of $|\psi_{n, c}(1)|^2$ in terms of Legendre elliptic  integrals, which is central here and leads to 
the formula \eqref{decay22}.

Let us come back to the present paper. When $n$ tends to $\infty$ with $c$ fixed, we recover the asymptotic behavior given by Widom and, as a corollary, we have  the following, which may be seen as a kind of quantitative Widom's Theorem.
\begin{corollary} \label{cor-intro}
Let $m>0$ be a positive real number and let $M>m,$  $\varepsilon>0$ be  given.  Then there exists a constant $A(\varepsilon, m, M)$ such that, for all $m\leq c\leq M\sqrt n$ and all $n$, we have the inequality
\begin{equation}
 \lambda_n(c)\leq A(\varepsilon, m, M) e^{\varepsilon n}\left( \frac {e c} {4(n+\frac 12)} \right)^{2n+1}.
\end{equation}
\end{corollary}
 We can give an explicit constant $A(\varepsilon, m, M)$. Also, we will show
 that asymptotically,  $\widehat \lambda_n(c)=2\widetilde \lambda_n(c)$ is equivalent   to Widom's asymptotic formula \eqref{Widom_asymptotic}.
 \smallskip

The fact that we recover Widom's asymptotic behavior is already a good test of validity but we can go further numerically. In fact  we recover the exact equivalent given by Widom up to the factor $1/2$, which justifies the approximation of  $\lambda_n(c)$ by $\widehat{\lambda_n}(c)= 2\widetilde{\lambda_n}(c)$ instead of $\widetilde{\lambda_n}(c)$, at least for large values of $n$.

\begin{remark}
 Numerical experiments  show that the approximation of $\lambda_n(c)$ by $\widehat{\lambda_n}(c)= 2\widetilde{\lambda_n}(c)$ is  surprisingly accurate.  It is striking that the values of ${\displaystyle |\widehat{\mu_n}(c)|=\sqrt{\frac{2\pi}{c}\widehat{\lambda_n}(c) }}$ coincide with   the values of  $|\mu_n(c)|$ that have been computed in\cite{Osipov-Rokhlin}, with  relative errors that are less than $3\%$. Numerical tests indicate that this relative error bound holds true as soon as $|\mu_n(c)|\leq 0.15.$ Moreover, the smaller the value of $|\mu_n(c)|,$ the smaller is the corresponding relative error. For example, for $c=10 \pi$ and $n=90,$ we have found that 
 $|\mu_n(c)|\approx 8.64288E-57$ and ${\displaystyle \frac{|\widehat{\mu_n}(c)-\mu_n(c)|}{|\mu_n(c)|}\approx 7.71E-05.}$
 \end{remark}

We try to explain the factor $2$ in the expression of  $\widehat{\lambda_n}(c),$ which is of course very small compared to the accumulated errors in the theoretical approach. 
Let us mention that another method to approximate the values $\lambda_n(c)$ has been used by Osipov in \cite{Osipov2}. The estimates given in his paper are of different nature and do not propose such a simple and accurate formula. In addition, he mainly considers values of $n$ such that $\frac{\pi n}{2}-c$ is smaller than some multiple of $\ln c$. At this moment both works may be seen complementary. But we underline the fact that numerical tests validate the accuracy of the approximant \eqref{decay22}   even when $c$ is close to the critical value, while our theoretical approach is not  sufficient to do it.
\medskip

This work is organized as follows. In Section 2,  we list some  estimates of the PSWFs and their associated eigenvalues $\chi_n(c)$.   In Section 3, we prove a sharp exponential decay rate
of the eigenvalues $\lambda_n(c)$ associated with the integral operator $\mathcal Q_c.$
 In Section 4,  we provide the reader
with  some numerical examples that illustrate the different
results	 of this work.

 We will systematically   skip the parameter  $c$ in $\chi_n(c)$ and $\psi_{n,c}$, when there is no doubt on the value of the bandwidth.
 We then note  $q= c^2/\chi_n$ and skip both parameters $n$ and $c$ when their values are obvious from the context. 

\section{Estimates of PSWFs and eigenvalues $\chi_n(c).$}
 Here we first list   some classical as well as some recent results on PSWFs and their eigenvalues $\chi_n$,
 then we push forward the methods and adapt them to our study. We systematically use the same notations as in \cite{Bonami-Karoui2}.
It is well known that  the eigenvalues $\chi_n$ satisfy the classical inequalities
\begin{equation}
\label{bounds1-chi}
n (n+1)   \leq \chi_n \leq n(n+1)+c^2.
\end{equation}
Next, we recall the Legendre elliptic  integral of the first and second kind, that are  given respectively, by
\begin{equation}
\label{elliptic_integrals}
\mathbf K(k)=\int_0^1 \frac{ dt}{\sqrt{(1-t^2)(1-k^2 t^2)}},\quad \quad
 \mathbf E(k)=\int_0^1 \sqrt{\frac{1- k^2 t^2}{1-t^2}}\, dt, \quad  0\leq k\leq 1.
\end{equation}
 Osipov has proved in \cite{Osipov} that the condition $q=\frac{c^2}{\chi_n}<1$
 is fulfilled when $c<\frac{\pi n}2$, while it is not when $c> \frac{\pi (n+1)}2.$ This is part of the following statement, which gives precise lower and upper bounds of the
 quantity ${\displaystyle \sqrt{q}= \frac{c}{\sqrt{\chi_n}}}$, see
 \cite{Bonami-Karoui2}.
\begin{lemma} \label{chi-between2}
For all $c>0$ and $n\geq 2$ we have
\begin{equation}\label{ineqPhi}
\Phi \left(  \frac {2c}{\pi(n+1)}\right) < \frac c{\sqrt{\chi_n}} < \Phi \left(\frac {2c}{\pi n}\right),
\end{equation}
where $\Phi$ is the  inverse of the function $k\mapsto \frac k{\mathbf E(k)}=\Psi(k),\,\, 0\leq k\leq 1.$
\end{lemma}
We refer to \cite{Bonami-Karoui2} for the proof of the previous lemma, but we add some comments. The inequalities \eqref{ineqPhi}
are equivalent  to the fact that
\begin{equation}\label{withE}
 \frac {\pi n}{2\E(\sqrt q)}<\sqrt{\chi_n}<  \frac {\pi(n+1)}{2\E(\sqrt q)}.
\end{equation}
The left hand side is due to Osipov \cite{Osipov}.
Note that $\Phi(0)=0$ and $\Phi(1)=1.$ Also, we should mention that 
\begin{equation}\label{derivé_Psi}
0\leq \Psi'(x)= \frac{\mathbf E(x) - x \mathbf E'(x)}{(\mathbf E(x))^2}=\frac{\mathbf K(x)}{(\mathbf E(x))^2},\quad 0\leq x <1,
\end{equation}
\begin{equation}\label{bounds_Phi'}
0\leq \Phi'(x)=\frac{(\mathbf E(\Phi(x)))^2}{\mathbf K(\Phi(x))}\leq \frac{(\mathbf E(0))^2}{\mathbf K(0)}=\frac{\pi}{2},\quad 0\leq x <1.
\end{equation}
Hence, $\Phi$ is an increasing function on $[0,1].$ Moreover, since ${\displaystyle \frac 2\pi\leq\frac{1}{\mathbf E(x)}\leq 1,}$ then  we have
$$\frac {2x}{\pi}\leq \Psi(x)\leq x.$$
Applying the function $\Phi$ to the previous inequalities, one  gets 
\begin{equation}\label{bounds_Phi}
x\leq \Phi(x)\leq \frac{\pi x}{2} , \quad 0\leq x\leq 1.
\end{equation}

We will use bounds for $\psi_{n, c}$ given in \cite{Bonami-Karoui2}, which have been established under the condition that $(1-q)\sqrt{\chi_n}>\kappa\geq 4$. We leave some flexibility for the choice of the constant $\kappa$ and do not restrict to the choice $\kappa=4$ as in \cite{Bonami-Karoui2}. We will only need estimates at $1$, which we give in the following lemma with a slightly different form compared to \cite{Bonami-Karoui2}.

\begin{lemma} \label{equivalence}
Let  $n\geq 3$. We assume that the condition
\begin{equation}\label{versusChi}
(1-q)\sqrt {\chi_n(c)}>\kappa
\end{equation} is satisfied for some $\kappa\geq 4$. Then,  there exists a constant $\delta(\kappa)$ (independent of $c$ and $n$) such that  one has  the following bounds for $(\psi_{n, c}(1))^2.$
\begin{equation}\label{boundsA}
 \frac{\pi \sqrt{\chi_n}}{2 \mathbf  K(\sqrt{q})}\left(1-\delta(\kappa) \, \varepsilon_n\right) \leq (\psi_n(1))^2 \leq  \frac{\pi \sqrt{\chi_n}}{2 \mathbf  K(\sqrt{q})} \left(1+\delta(\kappa)\, \varepsilon_n\right),\qquad \epsn=\left((1-q)\sqrt{\chi_n}\right)^{-1}.
\end{equation}\end{lemma}
We refer to  \cite{Bonami-Karoui2}, Theorem 2, for the  proof. Explicit values for the constant $\delta (\kappa)$  can also be deduced from \cite{Bonami-Karoui2}. We can choose
\begin{equation}\label{delta}
\delta(\kappa)=\eta\left( 2+\frac \eta{\kappa}\right), \qquad \eta = C(\kappa)\left(\frac{\beta }{1+(1- \kappa^{-1}\beta)^{1/2}}+\sqrt 2 \alpha (1+\alpha \kappa^{-1})\right)
\end{equation}
with $C(\kappa)^{-1}= (1- \kappa^{-1}\beta)^{1/2}-\sqrt 2 \alpha \kappa^{-1}(1+\alpha \kappa^{-1}),\,\,  \alpha=1.5,\,\, \beta=0.35.$\\

In any case, we see that the theoretical values of $\delta(\kappa)$  are larger than $4.73$.  We find approximatively $\delta(4)\approx 77.2$, $\delta(12)\approx 7.6$.  Numerical tests (see Example  1 in Section 4) indicate that the numerical quantities $\delta(\kappa)$, for which one has equality in \eqref{delta},  are much smaller.

\smallskip
In \cite{Bonami-Karoui1}, we have proved  that 
\begin{equation}
|A| =|\psi_{n, c}(1)|\chi_n(c)^{-1/4}\leq 2\qquad \mbox{for} \quad c\leq \frac{\pi(n+1)}{2}.
\end{equation}
So in particular the right hand side bound of \eqref{boundsA}  is not accurate when $\kappa$ is small.
Lemma \ref{equivalence} expresses the fact that, under some condition depending on a parameter $\kappa$, we have
\begin{equation}
\label{approxpsi(1)}
(\psi_{n,\tau}(1))^2\approx \frac{\pi\sqrt{\chi_n(\tau)}}{2 \mathbf  K(\sqrt{q(\tau)})}=\frac{\pi\tau}{2 \sqrt{q(\tau)}\mathbf  K(\sqrt{q(\tau)})}.
\end{equation}
The previous formula for the approximation of the quantity $\psi_{n,c}(1)$ still requires the approximation of  $\sqrt{q(\tau)}.$
For this last quantity, we already have the double inequalities \eqref{ineqPhi}. We may write
\begin{equation}\label{approxq}
\sqrt{q(\tau)}\approx\sqrt{\widetilde q(\tau)}= \Phi\left(\frac{2\tau}{\pi (n+1/2)}\right).
\end{equation}
An  error bound of the  previous approximation formula is given in  \cite{Bonami-Karoui2} by 
\begin{equation}\label{error-q}
| \sqrt{q(c)}- \sqrt{\widetilde q(c)}| \leq  \frac{c}{2\sqrt{\chi_n}\sqrt{\widetilde \chi_n}}.
\end{equation}
 Also, in \cite{Bonami-Karoui2}, we have given an explicit 
 formula for the approximation of $\sqrt{\chi_n(c)}$ together with its associated error,
 \begin{equation}
 \label{approxchi}
 \sqrt{\chi_n(c)}\approx \sqrt{\widetilde \chi_n(c)} =  \frac{c}{\Phi\left(\frac{2c}{\pi (n+1/2)}\right)},\quad n\geq \frac{2c}{\pi},
 \end{equation}
\begin{equation}\label{error-chi}
 \left|\sqrt{\chi_n(c)}-\sqrt{\widetilde \chi_n(c)}\right|\leq \frac{1}{2}.
\end{equation}
We should mention that in \cite{Bonami-Karoui2}, we have further improved the error bounds \eqref{error-q}  and \eqref{error-chi}
in the case of large values of $n.$ The improved asymptotic error bounds are given as follows,
\begin{equation}\label{approxq-Chi}
|\sqrt{q(c)}-\sqrt{\widetilde q(c)}|\leq  \frac{c\kappa}{(1-q)\chi_n\sqrt{\widetilde \chi_n }},\qquad |\sqrt{\widetilde \chi_n(c)}-\sqrt{ \chi_n(c)}|\leq  \frac{\kappa}{(1-q)\sqrt{\chi_n }},
\end{equation}
for some constant $\kappa.$ Numerical evidence indicates that in practice, the actual  error of the  approximation scheme \eqref{approxq} is much 
smaller than the previous theoretical error bound, see example 2 of the numerical results section. \\

 We need to translate Condition \eqref{versusChi} in terms of the parameters $n,c$, which can be done by using
 [Proposition 4, \cite{Bonami-Karoui2}], where the following inequality has been given. 
For $n\geq 2$ and $q<1,$
\begin{equation}\label{NtoCh}
(1-q)\sqrt{\chi_n}\geq \frac {(n-\frac{2c}{\pi})-e^{-1}}{\log n +5},
\end{equation}
A further improvement of the previous inequality is given by the following lemma.
\begin{lemma}\label{comparison} Let $n\geq 3$, $q<1$ and $\kappa\geq 4$. Then  one of the following conditions,
\begin{equation}
c\leq n-\kappa,
\end{equation}
\begin{equation}\label{NtoCh2}
\frac{\pi n} {2}-c> \frac{\kappa }4 (\ln (n)+9),
\end{equation} implies the inequality  \eqref{versusChi}, that is,
\begin{equation*}
(1-q)\sqrt {\chi_n(c)}>\kappa.
\end{equation*}
Moreover, if we assume  that $c>\frac {n+1}2$, then the condition $\frac{\pi n} {2}-c> \frac{\kappa }4 (\ln (n)+6)$ is sufficient.
\end{lemma}
\begin{proof}
Let $\gamma=\frac{2c}{\pi n}$. It follows from \eqref{withE} that
\begin{equation}\label{intermediate}
1-\gamma < 1-\sqrt q+\frac{ \E(\sqrt q)-1}{\E(\sqrt q)}.
\end{equation}
We claim that
\begin{equation}
\label{EEq2}
\mathbf E(x)-1 \leq  (1-x^2)\left( \frac{1}{4}\ln \left(\frac 1{1-x^2}\right)+\ln 2\right).
\end{equation}
Let us assume this and go on with the proof. It follows that
\begin{equation}
1-\gamma< \frac{1-q}{\E(\sqrt q)} \left( \frac{1}{4}\ln \left(\frac 1{1-q}\right)+\frac {\E(\sqrt q)}{1+\sqrt q}+\ln 2\right).
\end{equation}
We then use the elementary inequality, valid for $0<s<1$,
$$ s\ln ( 1/s)\leq 1/n + s\ln (n/e).$$
It implies that $$1-\gamma-\frac 1{4n \E(\sqrt q)}<\frac{1-q}{\E(\sqrt q)} \left( \frac{1}{4}\ln (n/e)+\frac {\E(\sqrt q)}{1+\sqrt q}+\ln 2\right).$$
We use also   \eqref{withE} to conclude that
\begin{equation}\label{versus1-q}
(1-q)\sqrt{\chi_n}\geq \frac{\pi n}{2\E(\sqrt q)}(1-q)>\kappa,
\end{equation}
whenever
$$\frac{\pi n} {2}-c> \kappa \left(\frac 14 \ln (n/e)+\frac {\E(\sqrt q)}{1+\sqrt q}+\ln 2\right)+ \frac 1{4n}.$$
This is the case, in particular, when $\frac{\pi n} {2}-c> \frac \kappa 4 \left(\ln (n)+9\right)$, using the fact that $\frac {\E(\sqrt q)}{1+\sqrt q}\leq \frac \pi 2$.

The condition $c\geq \frac {n+1}2$ implies that $q>\frac 1\pi$. Then, by  using the value of $\E(\sqrt {\pi^{-1}})$, the constant $9$ in (\ref{NtoCh2}) can be replaced by $6$. It remains to prove \eqref{EEq2}. We write
\begin{eqnarray}
\mathbf E(x)-1 &\leq & (1-x^2)\int_0^1\frac 1{(\sqrt{1-x^2t^2}+\sqrt{1-t^2})}\frac{t\, dt}{\sqrt{1-t^2}}\label{intermediate2}\\
&=&\int_0^1\frac {ds}{(1-x^2+s^2x^2)^{\frac 12}+s}.
\end{eqnarray}
We cut the last integral into two parts. For the first one, from $\sqrt{1-x^2}$ to $1$, we replace the denominator by $2s$ and find the logarithmic term. For the second one, we replace the denominator by $\sqrt{1-x^2}+s$ and find $\ln 2$.
\end{proof}

We will need another inequality of the same type:
\begin{equation}\label{comparisonK}
1-\frac{2c}{\pi n}\leq 2(1-q)\K(\sqrt q).
\end{equation}
This is a consequence of \eqref{intermediate}, using the fact that $\E(x)-1\leq (1-x^2)\K(x)$, which comes directly from \eqref{intermediate2}.

\section{Sharp decay estimates of eigenvalues $\lambda_n(c).$}

In this section, we use some of the estimates we have given in the previous section and
we prove a sharp  super-exponential decay rate of  the eigenvalues $(\lambda_n(c))_n.$
We first recall that these $\lambda_n(c)$ are governed by the following differential equation,
see \cite{Fuchs} or the more recent reference \cite {Xiao},
\begin{equation}\label{eq1lambda_n}
\partial_c \ln \lambda_n(c)=\frac{2|\psi_{n,c}(1)|^2}{c}.
\end{equation}
As a consequence, for fixed $n$ there exists a unique value of $c$ for which $\lambda_n(c)=1/2.$ We denote this value of $c$ by $c_n^*$. We know  from \cite{Landau} that it can be bounded below and above, namely
 \begin{equation}
 \label{validityC0}
 \frac{\pi}{2}(n-1)\leq c^*_n\leq \frac{\pi}{2} (n+1)\quad\mbox{ with}\quad \lambda_n(c^*_n)=\frac{1}{2}.
 \end{equation}
 By combining (\ref{eq1lambda_n}) and (\ref{validityC0}), one gets
 \begin{equation}
 \label{eqq2lambda_n}
 \lambda_n(c)=\frac{1}{2} \exp\left(-2\int_c^{c^*_n} \frac{(\psi_{n,\tau}(1))^2}{\tau} \, d\tau \right).
  \end{equation}
  Let us recall the following definition.
   \begin{equation}\label{decay2}
\widetilde { \lambda_n}(c) =  \frac{1}{2} \exp\left(-\frac{\pi^2(n+\frac 12)}{2} \int_{\Phi\left(\frac{2c}{\pi (n+\frac 12)}\right)}^1 \frac{1}{t(\mathbf E(t))^2}\, dt \right).
\end{equation}

Our main result is the following theorem.
\begin{theorem}\label{Required0}
There exist three constants $\delta_1\geq 1, \delta_2, \delta_3, \geq 0$ such that, for $n\geq 3$ and $c\leq \frac{\pi n}2$,
\begin{equation}\label{required0}
\delta_1 ^{-1}n^{-\delta_2}\left(\frac c{c+1}\right)^{\delta_3}\leq \frac{\widetilde { \lambda_n}(c)}{ \lambda_n(c)}\leq \delta_1 n^{\delta_2}\left(\frac c{c+1}\right)^{-\delta_3},
\end{equation}
\end{theorem}
The factor $\frac c{c+1}$ can be replaced by $1$ when $c>1$ and replaced by $c$ when $c<1$. We have written the formula this way  to avoid to have to distinguish between the two cases, $c\geq 1$ and $0<c<1.$   It is simpler to write equivalent  inequalities for logarithms, which is done in the following proposition. We keep the same notations for constants, which are of course not the same. We note $\ln^+ (x)$  the positive part of the Logarithm, that is, $\ln^+ (x)=\max(0, \ln (x))$. The following theorem
   is required in the proof of the main Theorem \ref{Required0}.
\begin{theorem}\label{Required}
There exist three non negative constants $\delta_1, \delta_2, \delta_3$ such that, for $n\geq 3$ and $c\leq \frac{\pi n}2$, we have
\begin{equation}\label{required}
\int_c^{c^*_n} \frac{(\psi_{n,\tau}(1))^2}{\tau} \, d\tau =\frac{\pi^2(n+\frac 12)}{4} \int_{\Phi\left(\frac{2c}{\pi (n+\frac 12)}\right)}^1 \frac{1}{t(\mathbf E(t))^2}\, dt  +\mathcal  E,
\end{equation}
with
\begin{equation}
|\mathcal E|\leq \delta_1+\delta_2 \ln (n)+\delta_3 \ln^+ (1/ c).
\end{equation}\end{theorem}

Let us make some comments before starting the proof. At this moment the three constants are not sufficiently small and cannot be used reasonably to obtain  numerical values. But they can be computed and are not that enormous. There is no hope, of course, to have found an exact formula for $\lambda_n(c)$ and \eqref{decay2} gives only an approximation. But these theoretical approximation errors may be seen as a kind of theoretical validation of the quality of approximation of the $\lambda_n(c),$ which we test numerically in Section 4.
\smallskip

It has been observed by many authors, and predicted by the work of Landau and Widom \cite{LW}, that for fixed $c$ the eigenvalues $ \lambda_n(c)$ decrease first exponentially in some interval starting at $ [\frac{2c}{\pi}]+1$ with length a multiple of $\ln(c)$, then super-exponentially as in the asymptotic behavior given by Widom. This is what one observes in Formula \eqref{decay2}, but the error terms do not allow to observe the  decay rate
at the plunge region. In fact the tools that we use, that is, the lower and upper bounds for $\psi_{n,\tau}(1)^2$, are only valid for $c_n^* -\tau$ sufficiently large in terms of $\ln(n)$.

We try to have  small constants at each step but are certainly far from the best possible. We give an explicit bound for $\mathcal E$ in
\eqref{final}. The following notations will be used frequently in the sequel. We define
\begin{eqnarray}
I(a, b)&=&\int_a^{b} \frac{(\psi_{n,\tau}(1))^2}{\tau} \, d\tau.\\
\mathcal J (y) &=&\frac{\pi^2 }{4} \int_{\Phi\left(\frac{2y}{\pi }\right)}^{1} \frac{1}{t(\mathbf E(t))^2}\, dt
\end{eqnarray}

We should mention that the proofs of Theorems 2 and Theorem 3, require many steps, so
 we start by giving a sketch of these proofs.

\medskip

\noindent{\sl \bf Sketch of the proof of Theorem 3.}
\smallskip
We want to prove that
$$I(c, c_n^*)\approx (n+\frac 12)\mathcal J \left(\frac c{n+\frac{1}{2}}\right).$$
For this purpose,  we use the approximation of $\psi_{n,\tau}(1),$ given by Formula \eqref{approxpsi(1)}. This is valid 
 under a condition involving the parameter  $\kappa,$ and may be rewritten as $c<c_n^\kappa$ for some $c_n^\kappa$ that is close to $c_n^*$ by Lemma \ref{comparison}.  We deduce from Formula \eqref{approxpsi(1)} that
$$I(c, c_n^\kappa)\approx \int_c^{c_n^\kappa} \frac{\pi d\tau}{2 \sqrt{q(\tau)}\mathbf  K(\sqrt{q(\tau)})} .$$
Then Lemma \ref{chi-between2} will be interpreted as the fact that
$$\sqrt{q(\tau)}\mathbf  K(\sqrt{q(\tau)})\approx \Phi\left(\frac {2\tau}{\pi(n+\frac 12)}\right)\mathbf  K\circ\Phi\left(\frac {2\tau}{\pi(n+\frac 12)}\right).$$
It is then elementary to relate the new integral with the function $\mathcal J$ and finally find that
$$I(c, c_n^\kappa)\approx (n+\frac 12)\mathcal J \left(\frac c{n+\frac{1}{2}}\right).$$ It remains to bound the tails of the integrals  $I( c_n^\kappa, c_n^*)$, which we can do because the two values are sufficiently close.
\smallskip

Let us start the proof itself. We need a set of intermediate results that can be classified into three main steps.
The first step will concern the properties of the function $\mathcal J$. In the second step,  we give bounds of the  tails of the integrals.
Finally, in the third step, we use the results of the previous two steps and complete the proofs of  Theorems 2 and 3.
\medskip

\noindent{\sl \bf First step: Properties of $\mathcal J$.}
\smallskip

For an integer $l\geq 1,$ we define 
\begin{equation}\label{Jl}
\mathcal J_l(c)=\frac \pi 2\int_c^{\frac{\pi l}2}\frac{d\tau}{\Phi\left(\frac {2\tau}{\pi l}\right)\mathbf  K\circ\Phi\left(\frac {2\tau}{\pi l}\right)}.
\end{equation}
As it has been seen in the sketch, these integrals are clearly involved in the proof. We first see that they are related with $\mathcal J$.
\begin{lemma}
We have the identity
\begin{equation}
\mathcal J_l(c)= l\mathcal J(c/l). \label{identity}
\end{equation}
\end{lemma}
\begin{proof}
We  consider   the substitution
\begin{equation}
\label{subs1}
 s = \Phi\left(\frac{2\tau}{\pi l}\right), \qquad   \tau= \frac{\pi l}2 \Psi(s).
\end{equation}
We have already seen in \eqref{derivé_Psi} that
$ \Psi'(x)=\frac{\mathbf K(x)}{(\mathbf E(x))^2}.$
Hence, we have
$$ \mathcal J_l(c)= l\int_{\Phi(\frac{2c}{\pi l})}^1\frac{ ds }{s (\mathbf E(s))^2}= l \mathcal J(c/l).$$
\end{proof}

The following proposition gives us  upper and lower  bounds, as well as the asymptotic behavior of $\mathcal J$.
 \begin{proposition}
For $x\in (0, \pi/2)$, one has the upper and lower bounds
\begin{equation}\label{majJ}
\ln^+\left(\frac{1}{ x}\right)\leq \mathcal J(x)\leq \frac{\pi^2 }{4} \ln\left(\frac{\pi}{2 x}\right).
\end{equation}
Moreover, one can write
   \begin{equation}\label{at-infty}
 \mathcal J(x)=\frac{\pi^2 }{4}\int_{\Phi(2x/\pi)}^1\frac{dt}{t(\mathbf E(t))^2}= \ln\left(\frac{4}{ex}\right)+ \mathcal E',
 \end{equation}
 with $|\mathcal E'|\leq \frac{\pi^2 x^2}{8} $.
\end{proposition}
\begin{proof}
The first inequalities are an easy consequence of the bounds below and above of $\Phi,$ given by
(\ref{bounds_Phi}).
Let us prove \eqref{at-infty}. We first write, for $0<y<1$,
\begin{equation}
\frac{\pi^2 }{4}\int_y^1 \frac{dt}{t(\mathbf E(t))^2} +\ln (y)=\Delta-\int_{0}^y \frac{\frac{\pi^2}{4}-\E(t)^2}{t(\mathbf E(t))^2} \,
 dt =\Delta - I_1(y). \end{equation}
Here
$$\Delta= \int_{0}^1 \frac{\frac{\pi^2}{4}-\E(t)^2}{t(\mathbf E(t))^2} \, dt.$$
It is probably well-known that
\begin{equation}\label{Delta}
\Delta=\ln\left(\frac{4}{e}\right)
\end{equation}
but we did not find any reference. We will see it as a corollary of Widom's Theorem.
The integral $I_1(y)$ is bounded by $\frac{\pi^2 y^2}8$. This  is a consequence of the elementary inequalities
$$1\leq \mathbf E(s)\leq \frac{\pi}{2},\quad \frac \pi 2 -\E(s)\leq s^2\int_0^1\frac{t^2 \,dt}{\sqrt{1-t^2}}= \frac{\pi s^2} 4.$$
Let us now fix $y=\Phi(2x/\pi)$. At this point we have proved that
$$0\leq \ln \left( \frac x y \right)-\mathcal E'= I_1(y) \leq \frac{\pi^2 y^2}8.$$
From the inequalities
$$\frac{2y}\pi\leq\frac{2x}\pi=\frac y{\E(y)}\leq \frac{2y}\pi (1-\frac{y^2}{2})^{-1}\leq \frac{2y}\pi (1+y^2),$$
it follows that $0\leq  \ln \left( \frac x y \right)+y^2$. This concludes the proof of the proposition.
\end{proof}
 This proposition leads to the following corollary, where we recognize the equivalent given by Widom.
\begin{corollary}\label{Tilde}
We have the double inequality
\begin{equation}\label{tilde}
\frac{1}{2}\left(\frac{ec}{4(n+\frac 12)}\right)^{2n+1}e^{-\frac {\pi^2}4 \frac{c^2}{n+\frac 12}}\leq \widetilde{ \lambda_n}(c)\leq \frac{1}{2}\left(\frac{ec}{4(n+\frac 12)}\right)^{2n+1}e^{+\frac {\pi^2}4 \frac{c^2}{n+\frac 12}}.
\end{equation}
\end{corollary}

\begin{proof}
Just note that ${\displaystyle \widetilde { \lambda_n}(c)=\frac{1}{2}\exp\left(-(2n+1)\mathcal J(c/(n+1/2))\right) }$ and use (\ref{at-infty})
with $x=\frac{c}{n+1/2}.$
\end{proof}
Let us go back to quantities $\mathcal J_l$. It is a straightforward consequence of \eqref{identity}  that the quantity $\mathcal J_l(c)$ increases with $l$. The next lemma gives reverse inequalities.
\begin{lemma}
 We have the inequalities
 \begin{equation}\label{comparisonJ}
\mathcal J_{n+1}(c)-\frac {\pi^2}{8} \ln \left(\frac{\pi(n+1)}{2c}\right)-\frac {\pi^3}{16} \leq \mathcal J_{n+\frac 12}(c)\leq \mathcal J_n(c)-\frac {\pi^2}{8} \ln \left(\frac{\pi(n+\frac 12)}{2c}\right)+\frac {\pi^3}{16}.
 \end{equation}
\end{lemma}
\begin{proof}
We will only prove  one of the inequalities, the other one being identical. Elementary computations give
$$ \mathcal J_{n+1}(c)-\mathcal J_{n+\frac 12}(c)\leq \frac 12 \mathcal{J}\left(\frac{c}{n+1}\right)+\frac {\pi^2}4 (n+\frac 12)
\ln\left(\frac{\Phi\left(\frac{2 c}{\pi (n+\frac 12)}\right)}{\Phi\left(\frac{2 c}{\pi (n+1)}\right)}\right).$$
We use \eqref{majJ} for the first term.  The second one is bounded by
$$
\frac {\pi^2}4 (n+\frac 12)\frac{\Phi\left(\frac{2 c}{\pi (n+\frac 12) }\right)-\Phi\left(\frac{2 c}{\pi (n+1)}\right)}{\Phi\left(\frac{2 c}{\pi (n+1)}\right)}\leq \frac{\pi^3}{16}.$$
Indeed, this is a consequence of the fact that $\Phi'(x)\leq \pi/2$ and ${\displaystyle \frac{x}{\Phi(x)}\leq 1,\,}$ for $ 0<x\leq 1$.
\end{proof}

\medskip

\noindent{\bf Second step: tails of the integrals.}
\smallskip

We fix some constant $\kappa\geq 4$ (for instance $\kappa=12$) and we assume that $n\geq 2\kappa+1$. Then, we know from Lemma \ref{comparison},  that the condition \eqref{versusChi}, that is,
$$(1-q)\sqrt {\chi_n}>\kappa,$$ is satisfied for $c<\frac{n+1}{2}$.
Next, if we  define
 \begin{equation}
\label{validityC}
  c^\kappa_n =\max\left(\frac{\pi n}{2}  -\frac \kappa 4 (\ln (n)+6), \frac{n+1}{2}\right)
\end{equation}
then, we have the following lemma.
\begin{lemma}
For $n\geq 2\kappa +1,$ we have the inequality
\begin{equation}\label{Ilarge}
I(c_n^\kappa, c_n^*) \leq \pi \kappa \ln (n)+ 6\pi\kappa+2\pi^2.
\end{equation}
\end{lemma}
\begin{proof}
Recall that ${\displaystyle |\psi_{n,c}(1)|\leq 2 \chi_n^{1/4}}$ and  ${\displaystyle \sqrt{\chi_n(c)}\leq \frac{\pi}{2}(n+1)}$, see
\cite{Osipov}. Hence, we have   
$$|\psi_{n, \tau}(1)|^2\leq 4\sqrt{\chi_n(\tau)}\leq 2\pi(n+1).$$
Consequently, one gets 
\begin{eqnarray*}
\int_{c^\kappa_n}^{c^*_n} \frac{(\psi_{n,\tau}(1))^2}{\tau} \, d\tau &\leq& 2\pi(n+1) \ln \left (1+\frac {\frac \pi 2+\frac\kappa 4(\ln (n)+6)} {c^\kappa_n}\right).
\end{eqnarray*}
We conclude by using the fact that $c^\kappa_n\geq \frac{n+1}{2}$.
 \end{proof}
We conclude directly the proof of Theorem \ref{Required} in the case where  $n\geq 2\kappa +1$ and $c <c_n^{\kappa}.$ It suffices to combine the results of Proposition 1 and the previous lemma, and get the desired inequalities
\begin{equation}
-\frac{\pi^2}{16}( \kappa \ln (n)+ 6\pi\kappa+\pi)\leq I(c, c^*_n)-(n+\frac 12)\mathcal J\left(\frac{c}{n+\frac 12}\right)\leq \pi \kappa \ln (n)+ 6\pi\kappa+2\pi^2
\end{equation}

We also conclude that  Theorem \ref{Required} and Theorem \ref{Required0} still hold  for the finite number of missing values of $n$, that is, $n\leq 2\kappa +1$. There is no problem to have upper bounds and lower bounds that do not depend on $c$ for $c<1$. From Corollary  \ref{Tilde}, we have a precise estimate in terms of $c^{2n+1}$ for $\widetilde{\lambda_n}(c)$. The same is given for $\lambda_n(c)$ by the following lemma.
\begin{lemma}
 Assume that $n\geq 1$ is fixed and let $0<c<1$. Then, there exist two constants $\delta(n), \delta'(n)$ such that
 \begin{equation}\label{chi}
 \delta(n) \, c^{2n+1}\leq \lambda_n(c)\leq  \delta'(n)\,  c^{2n+1}.
 \end{equation}
\end{lemma}
\begin{proof}
We first note that  $I(1, c^*_n)\leq I(1, \frac{\pi(n+1)}2)$. We recall that on this interval, we have the inequality $|\psi_{n, \tau}(1)|^2\leq 4\frac{\pi(n+1)}2$. So $I(1, c^*_n)\leq 2\pi (n+1) \ln(\frac{\pi(n+1)}2)$.
Inside the integral defining $I(c, 1),$
we use  the following inequality, that may be found in \cite{Bonami-Karoui2},
\begin{equation}
\left| |\psi_{n, \tau}(1)|-\sqrt{n+\frac 12}\right|\leq \frac{\tau^2}{\sqrt{3(n+1/2)}}\leq \frac{\tau^2}2.
\end{equation}
So $\left|I(c, 1)-(n+\frac 12)\ln\left(\frac 1c\right)\right|\leq 1$, from which we conclude.
\end{proof}
\smallskip
It remains to prove Theorem 2 and Theorem 3 when $c> c_n^{\kappa}$ and $n\geq 2\kappa +1.$

\noindent{\bf Third step: Proofs of Theorems 2 and 3.}
\smallskip

We fix $\kappa>4$.  Because of the previous steps,  we will only need to study the cases 
$$ n\geq 2\kappa +1 \qquad \qquad c<c_n^{\kappa}= \max\left(\frac{\pi n}{2}  -\frac \kappa 4 (\ln (n)+6), \frac{n+1}{2}\right).$$
In view of \eqref{required}, we want to bound the quantity 
$$\mathcal E=I(c, c_n^*)-\left(n+\frac 12\right)\mathcal J\left(\frac{c}{n+\frac 12}\right).$$
We have already given a bound to a first error term
$$\mathcal E_1=I(c, c_n^*)-I(c, c_n^\kappa).$$
Because of \eqref{Ilarge}, we know that
\begin{equation}\label{E1}
0\leq \mathcal E_1\leq \pi \kappa \ln (n)+ 6\pi\kappa+2\pi^2.
\end{equation}
Next, the conditions on $\kappa$ allow us to use the double inequalities \eqref{boundsA}. Namely,
\begin{equation}
 \label{psi1}
 \left(\psi_{n,\tau}(1)\right)^2 = \frac{\pi}{2 \mathbf K(\sqrt{q})} \sqrt{\chi_n(\tau)} +\mathcal{R}(\tau),\quad
 |\mathcal{R}(\tau)|\leq \frac{\delta(\kappa )}{(1-q(\tau)) \mathbf K(\sqrt{q(\tau)})},\quad 0\leq \tau \leq c^{\kappa}_n.
 \end{equation}
This leads to a second error, 
$$\mathcal E_2=I(c,c_n^{\kappa})-\frac{\pi}{2}\int_c^{c_n^{\kappa}}\frac{d\tau}{\sqrt{q(\tau)} \mathbf K(\sqrt{q(\tau)})},$$ which is bounded 
as follows,
$$|\mathcal E_2|\leq\delta(\kappa)\int_c^{c_n^\kappa}\frac 1{(1-q(\tau)) \mathbf K(\sqrt{q(\tau)})}\frac{ d\tau}\tau.$$
We then use the following lemma.
 \begin{lemma}
 We have the inequality
 \begin{equation}
 |\mathcal E_2|\leq 2\delta (\kappa)\left((1+\frac{\pi \kappa}{4} ) \ln(n)+\ln^+\left(\frac{1}{c}\right)+\frac{3\pi \kappa}{2}\right).
 \end{equation}
 \end{lemma}
 \begin{proof}
By \eqref{comparisonK}, we know  that
 $$2(1-q(\tau)) \mathbf K(\sqrt{q(\tau)})\geq 1- \frac{2 \tau}{\pi n}.$$
 So we have the inequality
 $$ |\mathcal E_2|\leq 2\delta (\kappa) \int_{\frac{2 c}{\pi n}}^{\frac{2 c_n^\kappa}{\pi n}}\frac{ ds}{(1-s)s}
 \leq 2\delta (\kappa)\left(\ln \left(\frac nc\right)+ \ln\left(\frac 1{1-\frac{2 c_n^\kappa}{\pi n}}\right)\right),$$
and we conclude at once.
\end{proof}

It remains to consider the main term, that is,
\begin{equation}
I_{\rm main}(c, c_n^\kappa)=\frac{\pi}{2}\int_c^{c_n^\kappa}\frac{\sqrt{\chi_n(\tau)}}{ \mathbf K(\sqrt{q(\tau)})} \frac{d\tau}{\tau}=\frac{\pi}{2}\int_c^{c_n^\kappa}\frac{d\tau}{ \sqrt {q(\tau)}\mathbf K(\sqrt{q(\tau)})} .
\end{equation}
We use the monotonicity properties of $\sqrt {q(\tau)}\mathbf K(\sqrt{q(\tau)})$ , namely
$$\Phi\left(\frac{2 \tau}{\pi (n+1)}\right)\mathbf K\circ\Phi\left(\frac{2 \tau}{\pi (n+1)}\right)\leq\sqrt {q(\tau)}\mathbf K(\sqrt{q(\tau)})
 \leq \Phi\left(\frac{2 \tau}{\pi n}\right)\mathbf K\circ\Phi\left(\frac{2 \tau}{\pi n}\right).$$
It follows that
$$ \mathcal{J}_{n}(c)- \mathcal{J}_{n}(c_n^\kappa)\leq I_{\rm main}(c, c_n^\kappa)\leq \mathcal{J}_{n+1}(c).$$
So the last error, 
$$\mathcal E_3=I_{\rm main}(c, c_n^\kappa)- \mathcal{J}_{n+\frac 12}(c)=I_{\rm main}(c, c_n^\kappa)- \left(n+\frac 12\right)\mathcal{J}\left(\frac{c}{n+\frac{1}{2}}\right),$$ satisfies the inequalities
$$\mathcal{J}_{n}(c)-\mathcal{J}_{n+\frac 12}(c)- \mathcal{J}_{n}(c_n^\kappa)\leq\mathcal E_3\leq \mathcal{J}_{n+1}(c)-\mathcal{J}_{n+\frac 12}(c).$$
It remains to use \eqref{majJ} and \eqref{comparisonJ}  to conclude. We finally find that
\begin{equation}\label{final}
| \mathcal E | \leq \pi \kappa \ln (n)+ 6\pi\kappa+2\pi^2+2\delta (\kappa)\left((1+\frac{\pi \kappa}{4} ) \ln(n)+\ln^+\left(\frac{1}{c}\right)+\frac{3\pi \kappa}{2}\right)+\frac {\pi^2}{8} \ln \left(\frac{\pi(n+\frac 12)}{2c}\right)+\frac {\pi^3}{16}.
 \end{equation}
 So we can take the following values for $\delta_1, \delta_2, \delta_3,$ that have been given in Theorem 3.
 \begin{eqnarray*}
 \delta_1&=&22+3\pi\kappa(2+\delta(\kappa))\\
 \delta_2&=&\frac {\pi^2}{8}+\pi \kappa+2\delta (\kappa)(1+\frac{\pi \kappa}{4} )\\
 \delta_3&=& \frac {\pi^2}{8}+2\delta (\kappa)(1+\frac{\pi \kappa}{4} ).
 \end{eqnarray*}
This concludes the proofs of Theorem \ref{Required} and Theorem \ref{Required0}.

Note that when $\kappa=12$ we find $\delta_2\approx 200$. We could have improved the sizes of the previous constants at each step, but not significantly. Numerical experiments indicate  that in practice, these constants are much smaller.
\smallskip

\bigskip

 From Theorem \ref{Required} and Corollary \ref{Tilde} we get the following corollary:
 \begin{corollary} There exist three constants $\delta_1\geq 1, \delta_2, \delta_3, \geq 0$ such that, for $n\geq 3$ and $c\leq \frac{\pi n}2$,
 \begin{equation}\label{best}
A(n, c)^{-1}\left(\frac{ec}{2(2n+1)}\right)^{2n+1}\leq  \lambda_n(c)\leq A(n,c)\left(\frac{ec}{2(2n+1)}\right)^{2n+1}.
\end{equation}
 with $$A(n, c)=\delta_1 n^{\delta_2}\left(\frac c{c+1}\right)^{-\delta_3}e^{+\frac {\pi^2}4 \frac{c^2}{n}}.$$
 \end{corollary}

Widom's Theorem says that $A(n, c)$ can be replaced by a quantity that tends to $1$ for $n$ tending to $\infty$. We cannot give such an asymptotic behavior at this moment, but we can estimate errors for fixed $c$ and $n$, which he does not. Remark that we have used the fact that $\Delta= \ln(4/e)$, see (\ref{Delta}), without proving it or giving a reference. This is a consequence of the asymptotic behavior found by Widom, which cannot  be valid at the same time as \eqref{best} if $e/4$ is replaced by another constant. 
 This implies in particular Theorem \ref{cor-intro}. It may be useful to give also the following  corollaries.
 \begin{corollary} There exist   constants $a>0$ and $\delta\geq 1$ such that, for $c\geq 1$ and $n>1.35\, c$, we have  
 \begin{equation}\label{best-bis}
 \lambda_n(c)\leq \delta e^{-an}.
\end{equation}
\end{corollary}
\begin{proof} The constant $1.35$ has been chosen so that $2\ln (\frac{4 n}{ec})>\frac {\pi^2c^2}{4n^2},$ 
which is the case when  $n>1.35 \, c$.
\end{proof}

One has as well a critical super-exponential decay rate given by the following lemma. 

 \begin{corollary} For any $0\leq a <\frac{4}{e},$ there exists a constant $M_a$ such that for any 
 $c\geq 1,$ we have  
 $$\lambda_n(c)\leq e^{-2n\log\left(\frac{a n}{c}\right)},\quad\forall\,\, n\geq c M_a.$$ 
 Moreover, for any $b>\frac{4}{e},$ there exists a constant $M_b$ such that for any 
  $c\geq 1,$ we have  
  $$\lambda_n(c)>  e^{-2n\log\left(\frac{b n}{c}\right)},\quad\forall\,\, n\geq c M_b.$$ 
\end{corollary}

The above corollary is a precise answer to Boyd's question on the super-exponential decay rate of the $\lambda_n(c),$ see \cite{Boyd2}.\\

\noindent
{\bf Final discussion and comments:}\\

\noindent
We should mention that one of the problems of our method of approximation of the eigenvalues $\lambda_n(c)$ is the fact that it  cannot be good for $(1-q)\sqrt{\chi_n}$ too close to $0$, while our technique of proof starts from the writing of $\ln(\lambda_n(c)$ as an integral from $c$ to  $c_n^*$. We have seen that asymptotically, for $c$ fixed and $n$ tending to $\infty$, we recover up to a factor of $1/2,$
the asymptotic behavior given by Widom, see Corollary \ref{Tilde}. The asymptotics for $n$ fixed and $c$ tending to $0$ is also well-known, see for example \cite{Xiao}. It may be written as
$$\lambda_n(c)\sim \left(\frac{ec}{4(n+\frac 12)}\right)^{2n+1}W_n$$
with $W_n$ that does not depend on $c$ and  tends to $1$ when $n$ tends to $\infty$. Because of this, we propose also the approximation of $\lambda_n(c)$ given by
\begin{equation}\label{hat}
\widehat{\lambda_n}(c)=2 \widetilde{\lambda_n}(c)=\exp\left(-\frac{\pi^2}{2}(n+\frac 12)\int_{\Phi(\frac{2c}{\pi(n+1/2)}}^{1}\frac{dt}{t\E(t)^2}\right).
\end{equation} 
Note that either one of 
$\widetilde{\lambda_n}(c)$ or $\widehat{\lambda_n}(c)$ can be used to get the precise super-exponential decay rate of the $\lambda_n(c).$
Moreover, both formulae can be tested for the approximation of the $\lambda_n(c).$ Nonetheless, numerical experiments show that the approximation by $\widehat{\lambda_n}(c)$ is surprisingly good for $c, n$ large. For smaller values (and in particular for small values of $(1-q)\sqrt{\chi_n}$), the approximation by $\widetilde{\lambda_n}(c)$ is better.

At this moment, we do not have a theoretical justification of this, apart from the asymptotic behavior of $\lambda_n(c)$. A tentative proof may start by writing $\lambda_n(c)$ with an integral from $0$ to $c$, instead of an integral from $c$ to $c_n^*$. Unfortunately, the singularity at $0$ of the integral makes estimates difficult and the idea of starting at $c_n^*$ instead of $0$ has been central here in order to benefit from the estimates on $\psi_n(1)$.

We do not give a formal proof but rather  some heuristic arguments. Heuristically , for $c'<c$, we have 
$$\ln\left(\frac{\lambda_n(c)}{\lambda_n(c')}\right) \approx \frac{\pi^2}{2}(n+1/2)\int_{\Phi(\frac{2c'}{\pi(n+1/2)})}^{\Phi(\frac{2c}{\pi(n+1/2)})}\frac{dt}{t\E(t)^2}.$$ Also, because of the asymptotic behavior of $\lambda_n(c')$ for $c'$ very close to $0$ and $n$ large enough, we have that
$$\ln\left(\frac{1}{\lambda_n(c')}\right) \approx \frac{\pi^2}{2}(n+1/2)\int_{\Phi(\frac{2c'}{\pi(n+1/2)})}^{1}\frac{dt}{t\E(t)^2}.$$ As a consequence of these two approximations, we have
$$\ln\left(\frac{1}{\lambda_n(c)}\right) \approx \frac{\pi^2}{2}(n+1/2)\int_{\Phi(\frac{2c}{\pi(n+1/2)})}^{1}\frac{dt}{t\E(t)^2}$$
as long as the approximation of the values of $|\psi_n(1)|$ are valid. That is, as long as $(1-q)\sqrt{\chi_n(c)}$ is not too small. 
The approximation $\lambda_n(c)$ by $\widehat{\lambda_n}(c)$ has been tested for  different values of $n$ and $c$ in the 
examples 3 and 4, below. From these simulations, we can think that the quantity ${\displaystyle \frac{\lambda_n(c)}{\widetilde{\lambda_n}(c)}}$ increases from 1 to 2 when $n$ goes from the beginning of the plunge region to infinity.

\section{Numerical results}

In this section, we illustrate the results of the previous  sections
by various numerical examples.

\vskip 0.5cm
\noindent
{\bf Example 1:} In this first example, we illustrate the fact that the actual values of the  constants $\kappa$ and $\delta(\kappa),$
given  by (\ref{versusChi}) and (\ref{boundsA}), respectively, are   much smaller than the theoretical values given in the proof of Lemma \ref{equivalence}. We are interested in these values for $n\geq 2c/\pi$.
For this purpose, we have considered the values of $c= m \pi, m=10, 20, 30,  40.$ We have used Flammer's method and computed
highly accurate values of $\chi_n(c)$ and $\psi_{n,c}(1).$ Then, we have  computed the smallest value of $\kappa,$ denoted by $\kappa_c$ and
ensuring the bounds  (\ref{boundsA}). Also, we have computed the corresponding values  $\delta(\kappa_c)$ so that
$(\psi_{n,c}(1))^2$ is equal to its upper bound given in (\ref{boundsA}).  It turns out that $\kappa_c,$ the critical value of $\kappa,$  is obtained for $n-$th eigenvalues $\chi_n(c)$ with
$n=n_c = [2c/\pi].$ Also, by considering various consecutive values of $n_c\leq n\leq  n_c+40$ and by computing the corresponding values of $\kappa$ and $\delta(\kappa),$ we found that the $\max \delta(\kappa)$ is of the same size as $\kappa_c.$
 Table 1 shows the  values of the critical values $\kappa_c$ and $\delta(\kappa_c)$ for the  different  values of the
bandwidth $c.$ Also, we give the values of $\max \delta(\kappa).$

\begin{center}
\begin{table}[h]
\vskip 0.2cm\hspace*{4cm}
\begin{tabular}{ccccc} \hline
 $c$ &$n_c$&$ \kappa_c$&$\delta(\kappa_c)$&$\max \delta(\kappa).$ \\   \hline
$10\, \pi$  &20    &  0.447   & 0.058  & 0.091    \\
$20\, \pi$ &40    & 0.413    & 0.051  & 0.084\\
$30\, \pi$ &60    & 0.394    & 0.047   & 0.080    \\
$ 40\, \pi$ &80   & 0.335     &0.025   & 0.048   \\
\hline
\end{tabular}
\caption{Critical values  of $\kappa,$  $\delta(\kappa)$ and $\max \delta(\kappa)$ for different values of $c.$}
\end{table}
\end{center}

\vskip 0.5cm
\noindent
{\bf Example 2:} In this example, we illustrate  our approximations of the quantity $\sqrt{q}$
 by $\sqrt{\widetilde q},$ given by  formula \eqref{approxq}. The accuracy of this  approximation is  critical 
 for proving the exact super-exponential decay rate of the $\lambda_n(c)$ by  our formula \eqref{decay2}. 
For this purpose we have considered different values of the bandwidth
$c$ and computed the previous approximations for different values of $n.$ These approximations are then compared with highly accurate 
counterparts obtained by the use of Flammer's method. The obtained numerical results are given by Table 2. 

\begin{center}
\begin{table}[h]
\vskip 0.2cm\hspace*{4cm}
\begin{tabular}{cccc} \hline
 $c$ &$n$&$\sqrt{\widetilde q}$&$\sqrt{q}$ \\   \hline
 & & & \\
$10$  &6       &0.995012670  & 0.99486271\\
      &10        &0.782942846  & 0.78302833 \\
      &15       &0.585651991  & 0.58583492   \\
 & & & \\
$25$ &16    &0.99062205     & 0.98924622\\
     &20    &0.90491661     & 0.90471915   \\
     &25    &0.79783057     & 0.79783979 \\
 & & & \\
$50$ &33 & 0.99501269    & 0.99430098\\
     &40  & 0.91050626    & 0.91045325 \\
     &50  & 0.80287160    & 0.80287326 \\
 & & & \\
$100$ &64 & 0.99705417    & 0.99669712 \\
      &80& 0.91330250    & 0.91328853   \\
      &100& 0.80540660    & 0.80540692   \\
\hline
\end{tabular}
\caption{Illustrations of the approximation formula \eqref{approxq}.}
\end{table}
\end{center}

\vskip 0.5cm
\noindent
{\bf Example 3:} In this example, we compare the  explicit formula 
  given by \eqref{hat} to compute highly accurate values of $\lambda_n(c)$. For this purpose, we have considered the values of $c=10 \pi, 20\pi, 30\pi$ and
computed   $\lambda_n(c)$ by using the method given in \cite{Karoui1}. Then, we have implemented
formula (\ref{hat}) in a Maple computing software code. Figure 1 (a), (b), (c) show the graph of  $\ln(\lambda_n(c))$ versus the graph of  $\ln(\widehat \lambda_n(c)),$ and $\ln(\lambda_n^W(c)),$ for the different values
of $c$ and $n.$ Here, $\lambda_n^W(c)$ is the Widom's asymptotic approximation of $\lambda_n(c),$ given by 
\eqref{Widom_asymptotic}. Also, we have plotted in Figure 2, the graphs of the corresponding values of
${\displaystyle \ln\left(\frac{\widehat{\lambda_n}(c)}{ \lambda_n(c)}\right)}.$ These figures illustrate the surprising precision of the explicit formula \eqref{hat} for computing approximate values of the $\lambda_n(c)$
which is numerically valid whenever $q<1.$ In particular, the numerical results illustrated by Figure 2, indicate that at least for moderate values of $c,$ and $q<1,$ the approximations of $\lambda_n(c),$ by either $\widetilde{\lambda_n}(c)$ or  $\widehat{\lambda_n}(c)$ are equal to 
$\lambda_n(c)$ up to a small multiplicative constant.

\begin{figure}[h]
{\includegraphics[width=15cm,height=6.2cm]{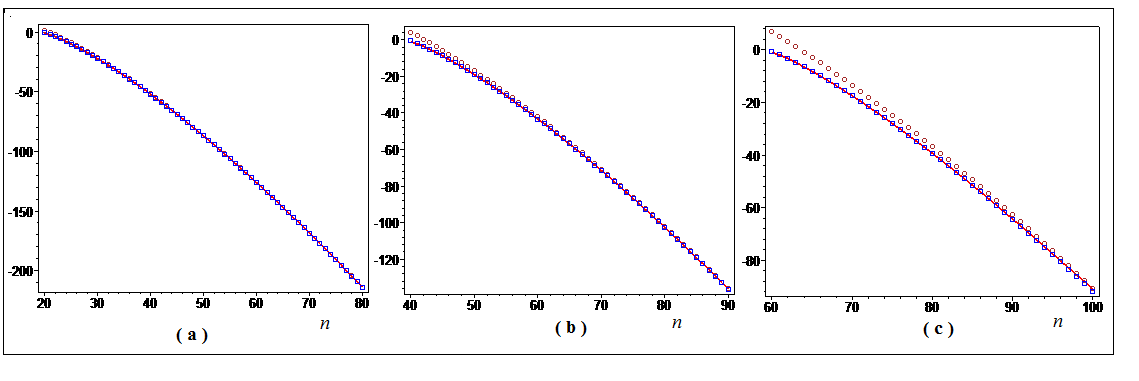}} \vskip
-0.5cm\hspace*{2cm}\caption{Graphs of
$\ln (\widehat{ \lambda_n}(c))$ (boxes), $\ln ( \lambda_n^W(c))$ (circles) and $\ln (\lambda_n(c))$ (red line) with $c=10 \pi$ for (a), $c=20\pi$ for  (b) and $c=30\pi$ for (c). }
\end{figure}

\begin{figure}[h]
{\includegraphics[width=15cm,height=5cm]{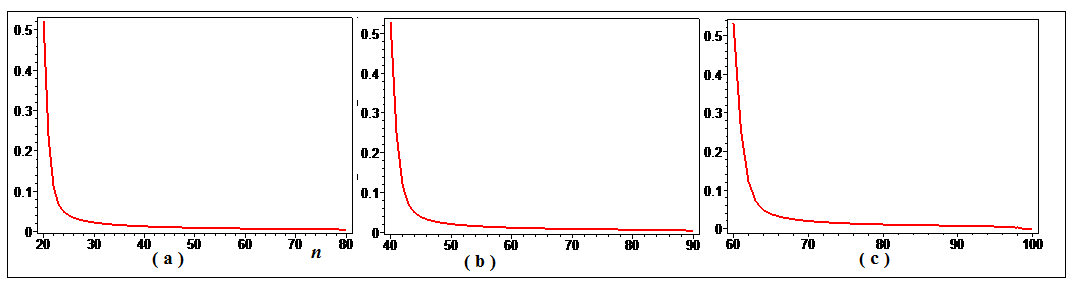}} \vskip
-0.5cm\hspace*{2cm}\caption{ Graphs of
$\ln\left(\frac{\widehat\lambda_n(c)}{{ \lambda_n(c)}}\right)$ with $c=10\pi$ for (a), $c=20\pi$ for (b) and $c=30\pi$ for (c).}
\end{figure}

\vskip 0.5cm
\noindent
{\bf Example 4:} In this example, we illustrate the accuracy of the approximation scheme \eqref{hat} in the cases where the bandwidth $c$
has relatively large or very large values. For this purpose, we have borrowed some  data given in Table 3 of \cite{Osipov-Rokhlin}, 
concerning the computation of $|\mu_n(c)|=\sqrt{\frac{2\pi}{c} \lambda_n(c)}.$ Note that in \cite{Osipov-Rokhlin} and in the present work, the roles of 
$\lambda_n$ and $\mu_n$ have been reversed.   The data provided by Osipov and Rokhlin
are  obtained by  the highly accurate  numerical  method for the computation of the $\lambda_n(c)$ developed by the authors    and described in  \cite{Osipov-Rokhlin}. These data are considered as references values and are  used for comparison purpose. Table 3 gives the values of 
$|\widehat \mu_n(c)|=\sqrt{\frac{2\pi}{c} \widehat\lambda_n(c)},$ versus the corresponding references values.
The numerical results of Table 3 indicate that the accuracy of formula \eqref{hat} is not affected
by the large values of $c.$ 
Also, to check the validity condition of our explicit approximation formula of the $\lambda_n(c),$ for each couple $(c,n),$ we have provided
the corresponding approximations of $q$ and $(1-q)\sqrt{\chi_n},$ given by $\widetilde q$ and $(1-\widetilde q)\sqrt{\widetilde \chi_n}.$
Note that for moderate and large values of the quantity $(1-q)\sqrt{\chi_n},$ a satisfactory approximation of this latter is given by 
the approximation $(1-\widetilde q)\sqrt{\widetilde \chi_n}.$ In fact, from \eqref{error-q} and \eqref{error-chi}, we have 
\begin{eqnarray*}
\left| (1-q)\sqrt{\chi_n}-(1-\widetilde q)\sqrt{\widetilde \chi_n}\right|&\leq & (1-q) \left|\sqrt{\widetilde \chi_n}-\sqrt{\chi_n}\right|+|q-\widetilde q|\sqrt{\widetilde \chi_n}\\
&\leq & \frac{1}{2} (1-\widetilde q)+ \left|\sqrt{q}-\sqrt{\widetilde q}\right| \frac{c}{2\sqrt{\chi_n}}\leq \frac{3}{2} -\frac{\widetilde q}{2}.
\end{eqnarray*}

\begin{center}
\begin{table}[h]
\vskip 0.2cm\hspace*{2cm}
\begin{tabular}{cccccc} \hline
 $c$ &$n$&$\widetilde q$&$(1-\widetilde q)\sqrt{\widetilde \chi_n}$&$|\widehat \mu_n|$&$|\mu_n|$ \\   \hline
$250$  &179     &0.924218 & 19.707014  & 0.18948E-07 & $0.18854E-07$  \\
       &184     &0.903501 & 25.380432  & 0.16196E-09 & $0.16130E-09$   \\
       &188     &0.886848 & 30.038563  & 0.30609E-11 & $0.30500E-11$  \\
$1000$  &659    &0.981782 & 18.386116  &0.38402E-07 & $0.38241E-07$  \\
       &665     &0.976303 & 23.983045  & 0.44139E-09 & $0.43991E-09$   \\
       &671     &0.970675 & 29.764638  & 0.42935E-11 & $0.42815E-11$     \\
$16000$  &10213 &0.998985 & 16.244476  & 0.56758E-07 & $0.56568E-07$  \\
       &10222   &0.998614 & 22.190912  &0.52955E-09 & $0.52821E-09$                 \\
       &10231   &0.998232 & 28.312611  &0.42989E-11 & $0.42902E-11$                   \\
$128000$  &81518&0.999881 & 15.293549  &0.42532E-07 & $0.42408E-07$  \\
       &81529   &0.999834 & 21.234778  &0.39992E-09 & $0.39906E-09$                 \\
       &81539   &0.999791 & 26.766672  &0.51858E-11 & $0.51768E-11$                   \\
$10^6$  &636652 &0.999986 & 13.738235  &0.51646E-07 & $0.51504E-07$  \\
       &636665  &0.999980 & 19.666621  &0.49076E-09 & $0.48980E-09$                 \\
       &636677  &0.999975 & 25.260364  &0.60652E-11 & $0.60558E-11$                   \\
\hline
\end{tabular}
\caption{Illustrations of the approximation formula \eqref{hat} for large values of $c, n.$}
\end{table}
\end{center}

\noindent
{\bf Acknowledgement:} The authors thank very much the anonymous referee for the valuable comments and suggestions
that helped them to improve the revised version of this work. Special thanks of the second author  go to Laboratory MAPMO 
of the University of Orl\'eans
where part of this work has been done while he  was a visitor there.

\end{document}